\documentclass[10pt,a4paper,oneside]{amsart}
\newcounter{commentcounter}

\usepackage{amsmath} % Lots of maths functionality
\usepackage{amssymb} % Maths symbols
\usepackage{amsthm} % Maths environments: \begin{proof}, etc.
\usepackage{thmtools}
\usepackage{stmaryrd} % [[ brackets
\usepackage[english]{babel} % Language and hyphenation.
\usepackage[font=small,justification=centering]{caption} % More flexibility for captioning figures
\usepackage[nodayofweek]{datetime}
\usepackage[shortlabels]{enumitem} % Change enumeration labelling with \begin{enumerate}[a)] etc.
\usepackage[T1]{fontenc} % For font encoding, to allow accents, copy & paste, inequality signs, etc. to all work nicely.
\usepackage[utf8]{inputenc} % To be loaded after fontenc, also for encoding.
\usepackage{ifthen} % For Dani's \begin{com} environment and \numberedtheorem
\usepackage{mathabx} % Contains \pnest symbol and more. Clashes with accents for \ring
\usepackage{mathtools} % Uses amsmath, fixes quirks and adds functionality.
\usepackage[dvipsnames]{xcolor} % Allow links to have colour. Needs to be before hyperref and tikz-cd
\usepackage[pdftex,  colorlinks=true, pagebackref=true]{hyperref} % Makes references and citations into links
    \hypersetup{urlcolor=RoyalBlue, linkcolor=RoyalBlue,  citecolor=black}
\usepackage{setspace} % Allows \onehalfspacing etc. for changing gaps between lines
\usepackage{tikz-cd} % Commutative diagrams
\usepackage{xfrac} % Nicer quotients 
\usepackage[capitalize]{cleveref} % use \Cref{} for instead of X~\ref{} 

\renewcommand*{\backref}[1]{}
\renewcommand*{\backrefalt}[4]
{
    \ifcase #1
        No citation in the text.
    \or
        Cited on Page #2.
    \else
        Cited on Pages #2.
    \fi
}

\makeatletter
\@namedef{subjclassname@1991}{Mathematical subject classification 1991}
\@namedef{subjclassname@2000}{Mathematical subject classification 2000}
\@namedef{subjclassname@2010}{Mathematical subject classification 2010}
\@namedef{subjclassname@2020}{Mathematical subject classification 2020}
\makeatother

\newtheorem{thm}{Theorem}[section]
\newtheorem{lemma}[thm]{Lemma}

\newtheorem{prop}[thm]{Proposition}

\newtheorem{claim}[thm]{Claim}

\newtheorem*{thm*}{Theorem}
\newtheorem*{lemma*}{Lemma}

\newtheorem{thmx}{Theorem}
\newtheorem{corx}[thmx]{Corollary}

\newtheorem{qux}[thmx]{Question}

\theoremstyle{definition}
\newtheorem{defn}[thm]{Definition}
\newtheorem{notn}[thm]{Notation}
\newtheorem{remark}[thm]{Remark}

\newtheorem{example}[thm]{Example}

\newtheorem*{remark*}{Remark}
\newtheorem*{ex*}{Example}
\newtheorem*{acks}{Acknowledgements}
\newtheorem*{outl}{Outline}

\theoremstyle{plain}

    \newtheoremstyle{TheoremNum}
        {8.0pt plus 2.0pt minus 4.0pt}{8.0pt plus 2.0pt minus 4.0pt} %%% space between body and thm
        {\itshape} %%% Thm body font
        {-0.15cm} %%% Indent amount (empty = no indent)
        {\bfseries} %%% Thm head font
        {.} %%% Punctuation after thm head
        { }  %%% Space after thm head
        {\thmname{#1}\thmnote{ \bfseries #3}}%%% Thm head spec
    \theoremstyle{TheoremNum}

\newcommand*{\claimproofname}{My proof}

\DeclareMathOperator{\Aut}{\mathrm{Aut}}

\newcommand*{\frakg}{\mathfrak{g}}

\newcommand*{\fraku}{\mathfrak{u}}
\newcommand*{\frakb}{\mathfrak{b}}
\newcommand*{\frakr}{\mathfrak{r}}
\newcommand*{\frakz}{\mathfrak{z}}
\newcommand*{\frakd}{\mathfrak{d}}

\newcommand{\OO}{\mathrm{O}}

\newcommand{\zar}[1]{\overline{#1}^{\mathrm{Zar}}}

\DeclareMathOperator{\Ad}{\mathrm{Ad}}
\DeclareMathOperator{\ad}{\mathrm{ad}}

\def\Z{\mathbb{Z}}

\newcommand{\NN}{\mathbb{N}}
\newcommand{\ZZ}{\mathbb{Z}}

\newcommand{\RR}{\mathbb{R}}

\usepackage{tikz}
\usetikzlibrary{arrows,quotes}
\tikzstyle{blackNode}=[fill=black, draw=black, shape=circle]

\title[Quasihomomorphisms to real algebraic groups]{Quasihomomorphisms to \\ real algebraic groups}

\usepackage{comment}
\usepackage{todonotes}

\author{Sami Douba}
\address[S.~Douba]{Mathematisches Institut, Rheinische Friedrich-Wilhelms-Universit\"at Bonn, Endenicher Allee 60, 53115 Bonn, Germany}
\email{douba@math.uni-bonn.de}

\author{Francesco Fournier-Facio}
\address[F. Fournier-Facio]{Department of Pure Mathematics and Mathematical Statistics, University of Cambridge, UK}
\email{ff373@cam.ac.uk}

\author{Sam Hughes}
\address[S.~Hughes]{Mathematisches Institut, Rheinische Friedrich-Wilhelms-Universit\"at Bonn, Endenicher Allee 60, 53115 Bonn, Germany}
\email{sam.hughes.maths@gmail.com; hughes@math.uni-bonn.de}

\author{Simon Machado}
\address[S.~Machado]{Department of Mathematics, ETH Z{\"u}rich, Switzerland}
\email{simon.machado@math.ethz.ch}

\date{\today}
\begin{document}

\begin{abstract}
A quasihomomorphism is a map that satisfies the homomorphism relation up to bounded error. Fujiwara and Kapovich proved a rigidity result for quasihomomorphisms taking values in discrete groups, showing that all quasihomomorphisms can be built from homomorphisms and sections of bounded central extensions. We study quasihomomorphisms with values in real linear algebraic groups, and prove an analogous rigidity theorem.
\end{abstract}

\maketitle

\section{Introduction}

Let $G$ be a locally compact group. We say that a subset $D \subset G$ is \emph{bounded} if $D$ is contained in a compact subset of $G$. Let $\Gamma$ be a group.  A map $f \colon \Gamma \to G$ is a \emph{quasihomomorphism} if there exists a bounded subset $D_f \subset G$ such that 
\[f(xy) \in f(x)f(y) D_f \text{ for all } x, y \in \Gamma.\] 
Two quasihomomorphisms $f, g: \Gamma \rightarrow G$ are \emph{close} if there exists a bounded subset $D_{f, g} \subset G$ such that $f(x) \in g(x) D_{f, g}$ for all $x \in \Gamma$. Note that the notion of quasihomomorphism depends on the topology of $G$, but not on the topology of $\Gamma$, so we always view $\Gamma$ as a discrete group.

Our main result is a rigidity theorem for quasihomomorphisms with values in a real linear\footnote{In the sequel, the word ``linear'' will be implicit.} algebraic group, in the spirit of \cite{fujiwarakapovich}. We define a \emph{bounded modification} of a quasihomomorphism $f \colon \Gamma \to G$ to be a quasihomomorphism\footnote{Outside the case of abelian target, a map that is close to a quasihomomorphism need not be a quasihomomorphism (\cref{ex bounded distance}).} $f' \colon \Gamma' \to G'$, where $\Gamma' < \Gamma$ is a finite-index subgroup, $G' < G$ is a closed subgroup, and $f'$ is close to $f|_{\Gamma'}$. If $G$ has the structure of a real algebraic group, then we require moreover that $G'$ be an algebraic subgroup of $G$. A subset $X$ of a group $G$ is \emph{normal} if $g^{-1}Xg = X$ for all $g \in G$.

\begin{thmx}[Main theorem, succinct form]
\label{main theorem}

Let $f \colon \Gamma \to G$ be a quasihomomorphism, where $G$ is a real algebraic group. Then there exists a bounded modification $f' \colon \Gamma' \to G'$ such that $D_{f'}$ is contained in a bounded normal subset of $G'$.
\end{thmx}

There is a rich theory of quasihomomorphisms when the target is an abelian group, thanks to their many appearances throughout mathematics: stable commutator length \cite{calegari}, bounded cohomology \cite{frigerio}, one-dimensional dynamics \cite{ghys}, geometric group theory \cite{bestvina:fujiwara, manning}, knot theory \cite{malyutin} and symplectic geometry \cite{polterovichrosen}. Rigidity results for quasihomomorphisms can be traced back to work of Burger--Monod \cite{burgermonod1, burgermonod2} on bounded cohomology of higher-rank lattices. This motivated Ozawa \cite{Ozawa2011} to introduce property (TTT) as a tool to prove rigidity of quasihomomorphisms with amenable target. Dumas \cite{dumas2024quasihomomorphism} has since proved that lattices in higher-rank simple algebraic groups over local fields satisfy this property. Recently, the work of Hrushovski \cite{hrushovski} showed that quasihomomorphisms play a central role in model theory, especially in relation to approximate subgroups.

The result in this direction that is most relevant to the present work is due to Fujiwara--Kapovich \cite{fujiwarakapovich}, who studied the case in which the target is discrete (so that bounded sets are finite). They proved that if $f \colon \Gamma \to G$ is a quasihomomorphism with discrete target, then there exists a bounded modification $f' \colon \Gamma' \to G'$ such that $D_{f'}$ is \emph{central} in $G'$. Therefore, up to passing to a finite-index subgroup and changing by a bounded amount, one can construct $f$ from homomorphisms and sections of bounded central extensions. Accordingly, they call such quasihomomorphisms \emph{constructible}. 

In the non-discrete setting, one cannot expect such a strong rigidity statement, even for $G=\mathrm{SO}(2) \ltimes \RR^2$: see \cref{counterexample central defect}. As remarked by Hrushovski \cite[Section 5.17]{hrushovski}, the right generalisation of ``central'' is ``normal'' once one abandons the discrete world.

\begin{remark*}
    There appears to be some confusion in the literature about constructibility of quasihomomorphisms: \cite{rank1} claims the existence of non-constructible quasihomomorphisms in the case where $\Gamma$ is the fundamental group of a closed hyperbolic manifold and $G$ is a non-abelian unipotent real algebraic group. Below, we exhibit a counterexample to \cite[Theorem 4.28]{rank1}\footnote{Theorem 4.24 in the arXiv version}; see \cref{ex BV24} and \cref{rem BV24} for a discussion.
\end{remark*}

Quasihomomorphisms with normal defect play a special role in the literature; indeed, they are the ones that appear in relation to model theory \cite{hrushovski}. It turns out that such quasihomomorphisms are very restricted, in the general case of a locally compact target \cite[Proposition C.3]{hrushovski}. In our case we can push this even further, giving the following stronger version of \cref{main theorem}. Following Hrushovski \cite{hrushovski}, we say that a  simply connected  abelian algebraic normal subgroup $A \lhd G$ is \emph{rigid} if the homomorphism $G \to \Aut(A)$ induced by conjugacy has compact image.\footnote{Hrushovski also allows for $\ZZ^n$ factors, but these do not occur in the algebraic setting.} In this case, $A \cong \RR^n$ for some $n \in \NN$, and there is a Euclidean structure on $A$ such that the action of $G$ factors through $\OO(n)$.

\begin{thmx}[Main theorem, stronger form]
\label{main theorem strong}
    Let $f \colon \Gamma \to G$ be a quasihomomorphism, where $G$ is a real algebraic group. Then there exists a bounded modification $f' \colon \Gamma' \to G'$ with the following property:

    There exists a compact (algebraic) normal subgroup $C \lhd G'$ and a rigid abelian algebraic normal subgroup $A \lhd G'/C$ such that $D_{f'}$ is contained in the preimage of $A$ in $G'$.
\end{thmx}

It follows that the composition $\Gamma' \xrightarrow{f'} G' \to (G'/C)/A$ is a homomorphism. Therefore, up to passing to a finite-index subgroup and changing by a bounded amount, one can build $f$ from homomorphisms, sections of bounded \emph{abelian} extensions, and sections of compact extensions. Since bounded abelian extensions are governed by bounded cohomology, we obtain the following corollary.

\begin{corx}
\label{h2b}
    Let $\Gamma$ be a group with the property that for every finite-dimensional orthogonal representation $\pi$ of $\Gamma$ with no non-zero invariant vectors, the second bounded cohomology $H^2_b(\Gamma; \pi)$ vanishes. Then for every quasihomomorphism $f \colon \Gamma \to G$, where $G$ is a real algebraic group, there exists a bounded modification $f' \colon \Gamma' \to G'$, and a compact (algebraic) normal subgroup $C \lhd G'$ such that the composition $\Gamma' \xrightarrow{f'} G' \to G'/C$ is a quasihomomorphism with central defect.

    If moreover the second bounded cohomology with trivial coefficients $H^2_b(\Gamma;\RR)$ vanishes, then $f' \colon \Gamma' \to G'$ and $C \lhd G'$ can be chosen such that $\Gamma' \xrightarrow{f'} G' \to G'/C$ is a homomorphism.
\end{corx}

The first statement of \cref{h2b} applies to lattices in higher-rank simple groups, while the second statement applies only in the non-Hermitian case \cite{monodshalom}. Moreover, the second statement applies to amenable groups \cite{johnson}, certain lamplighters \cite{monod:F, elena}, Thompson's group $F$ and other groups of dynamical origin such as compactly supported transformation groups of Euclidean spaces \cite{monod:F, ccc} and certain natural colimit groups arising in the context of homological stability \cite{ccc}. In the terminology of \cite{fujiwarakapovich}, the last conclusion is that the quasihomomorphism $f'$ is an \emph{almost homomorphism}.

\medskip

Under stronger assumptions on the target group, we obtain severe restrictions on quasihomomorphisms, as demonstrated by the following result, which serves as a starting point for the proof of the main theorem.

\begin{thmx}
\label{simple case}
    Let $\mathbb{K}$ be a local field, let $G = {\bf G}(\mathbb{K})$, where ${\bf G}$ is a connected adjoint simple $\mathbb{K}$-group, and let $f \colon \Gamma \to G$ be a quasihomomorphism. If $f(\Gamma)$ is unbounded in $G$ and generates a Zariski-dense subgroup of $G$, then $f$ is a homomorphism.
\end{thmx}

\cref{simple case} holds also for non-Archimedean local fields, in particular it holds for $p$-adic groups. However, in the final step of our proof of the main theorem, we need the fact that a real algebraic group with a compact Zariski-dense subgroup is itself compact (see the footnote in \cref{endgame}), hence our restriction to the setting of real algebraic groups.

\medskip

A consequence of the proof of the main theorem is that the difference between \emph{normal} and {\em central} defect is only due to the presence of compact quotients, as in \cref{counterexample central defect}. Therefore, excluding compact quotients, we obtain the following direct analogue of \cite{fujiwarakapovich}.

\begin{thmx}
\label{no compact factors}
    Let $f \colon \Gamma \to G$ be a quasihomomorphism, where $G$ is a real algebraic group. Suppose that the Zariski-closure in $G$ of the group generated by $f(\Gamma)$ has no non-trivial compact algebraic quotients. Then there exists a bounded modification $f' \colon \Gamma \to G'$ such that $D_{f'}$ is central in $G'$.
\end{thmx}

Note that here we did not need to pass to a finite-index subgroup of $\Gamma$. More generally, in the proof of the main theorem, we need only pass to a finite-index subgroup to ensure that the Zariski-closure of the group generated by the image of the quasihomomorphism is Zariski-connected (\cref{connected}).

\medskip

We believe that the analogue of \cref{main theorem} should hold in the more general setting of real Lie groups. Together with the Gleason--Yamabe Theorem \cite{tao} (using \cref{compact extension}), this would then give a rigidity result for quasihomomorphisms with connected locally compact target. In the opposite direction, we do not know whether to expect a positive answer to the following:

\begin{qux}
\label{q:tdlc}
    Is there a quasihomomorphism $f \colon \Gamma \to G$, where $G$ is a totally disconnected, locally compact group, for which \cref{main theorem} does not hold? What if one assumes $\Gamma = \ZZ$?
\end{qux}

\begin{outl}
    We start with generalities on quasihomomorphisms in \cref{s:qhm}. Then in \cref{s:simple} we prove \cref{simple case}, which serves as a starting point for the proof of the main theorem. In \cref{s:lag}, we prove some technical results on real algebraic groups, most notably \cref{exp local homeo}. In \cref{s:proof} we prove the main theorem in all its various versions: \cref{main theorem}, \cref{main theorem strong}, \cref{h2b} and \cref{no compact factors}. Finally, in \cref{s:cex}, we prove \cref{counterexample central defect}, which shows that in our most general setting one cannot hope for a bounded modification with central defect.
\end{outl}

\begin{acks}
    SD and SH are supported by Deutsche Forschungsgemeinschaft (DFG, German Research Foundation) under Germany's Excellence Strategy - EXC-2047/1 - 390685813. FFF is supported by the Herchel Smith Postdoctoral Fellowship Fund.  SH is supported by a Humboldt Research Fellowship at Universit\"at Bonn. Various stages of  this research were conducted at the Forschungsinstitut f\"ur Mathematik (ETH Z\"urich); at the Isaac Newton Institute for Mathematical Sciences during the programme Operators, Graphs, Groups (EPSRC grant EP/Z000580/1s); and at the Simons Laufer Mathematical Sciences Institute in Berkeley, California, during the Spring 2026 semester (National Science Foundation Grant No. DMS-2424139); the authors acknowledge these institutions' hospitality and support.
\end{acks}

\section{Generalities on quasihomomorphisms}
\label{s:qhm}

\begin{notn}
\label{notation approx}

We work with right actions. Accordingly, conjugacy is denoted by $g^h = h^{-1} g h$, and commutators by $[g, h] = g^{-1} h^{-1} g h = g^{-1} g^h$.

Let $K$ be a bounded set. If $x\in yK$, then we write $x\approx_K y$. Below we will introduce shorthand notations to specify what the compact set is: see \cref{long product} and the proof of \cref{final modification}.

When $H$ is a subgroup of the algebraic group $G$, we denote by $\zar{H}$ the Zariski-closure of $H$ in $G$.
\end{notn}

\begin{defn}
    Let $f \colon \Gamma \to G$ be a map to a locally compact group $G$. The \emph{defect set} of $f$ is
    \[D_f \coloneqq \{ f(y)^{-1}f(x)^{-1}f(xy) \mid x, y \in \Gamma \}.\]
    We say that $f$ is a \emph{quasihomomorphism} if $D_f$ is bounded. The \emph{defect group} of $f$ is the topological closure of the group generated by $D_f$ and is denoted by $\Delta_f$.
\end{defn}

Using the notation above, $f$ is a quasihomomorphism if $f(xy) \approx_K f(x)f(y)$ for some bounded set $K$. Similarly, $f$ and $g$ are close if $f(x) \approx_K g(x)$.

\begin{lemma}~\cite[Section 2]{fujiwarakapovich}
\label{FK conjugacy}
    Let $f \colon \Gamma \to G$ be a function. Then
    \[D_f^{f(x)} \subset D_f^2 D_f^{-1}\]
    for all $x \in \Gamma$.
\end{lemma}

\begin{lemma}
\label{BOL D}
    Let $x \in \langle D_f \rangle$. Then $x^{f(\Gamma)}$ is bounded.
\end{lemma}

\begin{proof}
    By \cref{FK conjugacy}, we have
    \[x^{f(\Gamma)} \subset (D_f^{f(\Gamma)} \cup (D_f^{-1})^{f(\Gamma)})^k \subset (D_f^2 D_f^{-1} \cup D_f D_f^{-2})^k. \qedhere\]
\end{proof}

\begin{prop}
\label{defect group normal}

Suppose that $G$ is an algebraic group, and that $f \colon \Gamma \to G$ is a quasihomomorphism. Suppose that $f(\Gamma)$ generates a Zariski-dense subgroup of $G$. Then $\zar{\Delta_f} = \zar{\langle D_f\rangle}$ is a normal subgroup of $G$.
\end{prop}

\begin{proof}
\cref{FK conjugacy} implies that every element of $f(\Gamma)$ normalises $\langle D_f \rangle$, hence also $\Delta_f$. It follows that the group generated by $f(\Gamma)$ normalises $\Delta_f$, and hence also $\zar{\Delta_f}$. This is an algebraic condition, therefore by Zariski-density $G$ normalises $\zar{\Delta_f}$, too.
\end{proof}

\begin{lemma}
\label{long product}
Let $y_1, \ldots, y_m, z_1, \ldots, z_m \in \Gamma$ and let $f \colon \Gamma \to G$ be a quasihomomorphism. Then
\[f(y_1 z_1^{-1} \cdots y_m z_m^{-1}) \approx_{K_m} f(y_1) f(z_1)^{-1} \cdots f(y_m) f(z_m)^{-1},\]
where $K_m$ is the set $(D_f \cup D_f^{-1})$ to a power depending only on $m$.
\end{lemma}

\begin{proof}
For $m = 1$, first note that
\[f(1) = f(1 \cdot 1) \in f(1) f(1) D_f \quad\Rightarrow\quad f(1) \in D_f^{-1}.\]
Next,
\[f(1) = f(x x^{-1}) \in f(x) f(x^{-1}) D_f \quad\Rightarrow\quad f(x)^{-1} \in f(x^{-1}) D_f^2.\]
It follows that
\[f(y z^{-1}) \in f(y) f(z^{-1}) D_f \subset f(y) f(z)^{-1} D_f^{-2} D_f.\]
So $K_1 = (D_f \cup D_f^{-1})^3$ works.

Suppose that the statement is true up to $m$, and let $K_m$ be the corresponding bounded set. Then
\small
\begin{align*}
f(y_1 z_1^{-1} \cdots y_{m+1} z_{m+1}^{-1}) &\in f(y_1 z_1^{-1}) f(y_2 z_2^{-1} \cdots y_{m+1} z_{m+1}^{-1}) D_f \\
& \subset f(y_1) f(z_1)^{-1} K_1 f(y_2) f(z_2)^{-1} \cdots f(y_{m+1}) f(z_{m+1})^{-1} K_m D_f.
\end{align*}
\normalsize
Using \cref{FK conjugacy}, we can conjugate $K_1$ by appropriate elements in the image of $f$ while keeping the error set a power of $(D_f \cup D_f^{-1})$.
\end{proof}

\begin{remark}
\label{compact extension}

Let $\pi \colon G \to Q$ be quotient with compact kernel $C$. If $f \colon \Gamma \to G$ is a quasimorphism, then $D_{\pi f} = \pi(D_f)$, hence $\pi f$ is also a quasihomomorphism. Conversely, if $g \colon \Gamma \to Q$ is a quasihomomorphism, then and $\hat{g} \colon \Gamma \to G$ is a set-theoretic lift, then $D_{\hat{g}} \subset \pi^{-1}D_g$, hence $\hat{g}$ is also a quasihomomorphism.

It is also easy to see that the conclusion of our main theorem (in its various forms: \cref{main theorem}, \cref{main theorem strong} and \cref{no compact factors}) is preserved in passing from $f \colon \Gamma \to G$ to $\pi f \colon \Gamma \to Q$, and conversely in passing from $g \colon \Gamma \to Q$ to $\hat{g} \colon \Gamma \to G$.
\end{remark}

\section{The simple case}
\label{s:simple}

Throughout this section, we denote by $\mathbb{K}$ a local field, by ${\bf G}$ a connected adjoint simple algebraic $\mathbb{K}$-group such that $G := {\bf G}(\mathbb{K})$ is non-compact. 

\begin{defn}
    Given a pair $P^\pm$ of opposite proper parabolic subgroups of $G$ and $z^\pm \in G/P^\pm$, a sequence $g_n$ in $G$ is {\em $(z^+, z^-)$-contracting} if $g_n$ converges to the constant function $z^+$ uniformly on compact subsets of $(G/P^+) \setminus Z_{z^-}$, where $Z_{z^-} \subset G/P^+$ denotes the set of all points in $G/P^+$ that are opposite to $z^-$. 
\end{defn}

The following two lemmas are well known.

\begin{lemma}~\cite[Prop.~6.9]{KLPdynamics}\label{contractingsubsequence}
    Any unbounded sequence in $G$ possesses a subsequence that is $(z^+, z^-)$-contracting for some pair $P^\pm$ of opposite proper parabolic subgroups of $G$ and some $z^\pm \in G/P^\pm$.
\end{lemma}

\begin{lemma}~\cite[Prop.~3.3]{GGKW}
\label{reducetosln}
    Given any pair $P^\pm$ of opposite proper parabolic subgroups of $G$, one can find a representation $\tau: G \rightarrow \mathrm{SL}_d(\mathbb{K})$ and $\tau$-equivariant embeddings $\iota^\pm : G/P^\pm \rightarrow \mathrm{SL}_d(\mathbb{K})/P_1^\pm$, where $P_1^+$ (resp., $P_1^-$) denotes the stabiliser in $\mathrm{SL}_d(\mathbb{K})$ of a line (resp., a hyperplane) in $\mathbb{K}^d$, such that for any $z^\pm \in G/P^\pm$ and any $(z^+, z^-)$-contracting sequence $g_n$ in $G$, the sequence $\tau(g_n)$ is $(\iota^+(z^+), \iota^-(z^-))$-contracting. 
\end{lemma}

\begin{lemma}\label{parabolic}
    Let $h \in G$, let $P^\pm$ be a pair of opposite proper parabolic subgroups of $G$, and let $g_n$ be a $(z^+, z^-)$-contracting sequence in $G$ for some $z^\pm \in G/P^\pm$. If the sequence $g_n^{-1} h g_n$ is bounded in $G$, then $h z^+ = z^+.$
\end{lemma}

\begin{proof}
By Lemma \ref{reducetosln}, it suffices to consider the case where $G = \mathrm{SL}_d(\mathbb{K})$ and $P^+$ (resp., $P^-$) is the stabiliser in $G$ of the line $\mathbb{K}e_1$ (resp., the hyperplane $\mathbb{K}e_2 \oplus \cdots \oplus \mathbb{K}e_d$) in $\mathbb{K}^d$. Let $g_n = k_n a_n k_n'$ be the $KA^+K$ decomposition of $g_n$. We can assume up to subsequence that $k_n$ (resp., $k_n'$) converges, say to $k$ (resp., $k'$). Then $a_n$ is a diagonal matrix of the form $\mathrm{diag}(\lambda_{1,n}, \ldots, \lambda_{d,n})$, where $\lambda_{1,n} \geq \cdots \geq \lambda_{d,n} > 0$, and since $g_n$ is $(z^+, z^-)$-contracting, we have that  $\lambda_{1,n}/\lambda_{2,n} \rightarrow \infty$ and $z^+ = k \cdot \mathbb{K}e_1$. 

Since the sequence $g_n^{-1} h g_n = k_n'^{-1} (a_n^{-1}k_n^{-1}h k_na_n)k_n$ is bounded, so is the sequence $a_n^{-1}k_n^{-1}h k_na_n$. Now if $(x_{1,n}, \ldots, x_{d,n})^T$ is the first column of the matrix $k_n^{-1}h k_n$, then $(\frac{\lambda_{1,n}}{\lambda_{1,n}}x_{1,n}, \frac{\lambda_{1,n}}{\lambda_{2,n}}x_{2,n}, \ldots, \frac{\lambda_{1,n}}{\lambda_{d,n}}x_{d,n})^T$ is the first column of the matrix $a_n^{-1}k_n^{-1}h k_na_n$. Since $x_{i,n}$ are bounded, it follows that $x_{i,n} \rightarrow 0$ as $n \rightarrow \infty$ for $i=2, \ldots, d$. Since $k_n^{-1}h k_n \rightarrow k^{-1} h k$, it follows that $k^{-1} h k$ fixes $\mathbb{K} e_1$, and so $h$ fixes $k \cdot \mathbb{K}e_1 = z^+$.
\end{proof}

\begin{proof}[Proof of \cref{simple case}]
    Suppose that $D_f$ is not reduced to the identity and $f(\Gamma)$ is not bounded. We show that $f(\Gamma)$ is contained in a proper algebraic subgroup of $G$. We first note that it suffices to show that $D_f$ is contained in a proper algebraic subgroup $L$ of $G$. Indeed, by \cref{defect group normal} it follows that $f(\Gamma)$ is contained in the normaliser of $\zar{\Delta_f}$, which is a nontrivial proper algebraic subgroup of $G$.
    
    Now since $f(\Gamma)$ is not bounded in $G$ by assumption, we have by Lemma~\ref{contractingsubsequence} that there exists a pair of opposite proper parabolic subgroups $P^\pm < G$, $z^\pm \in G/P^\pm$, and a $(z^+, z^-)$-contracting sequence $f(\gamma_n)$ in $f(\Gamma)$. By \cref{BOL D}, the sequence $f(\gamma_n)^{-1} h f(\gamma_n)$ is bounded for each $h \in D_f$. It then follows from Lemma \ref{parabolic} that $D_f \subset \mathrm{Stab}_G(z^+)$. This completes the proof of \cref{simple case} since $\mathrm{Stab}_G(z^+)$ is a proper algebraic subgroup of $G$ (indeed, a conjugate of $P^+$).
\end{proof}

\subsection{Examples}

The simple case already exhibits two striking examples. Let $G = \mathrm{PSL}_2(\RR)$, and let $F < G$ be a non-abelian free subgroup.\footnote{Note that $\mathrm{PSL}_2(\mathbb{R})$, while not itself an affine algebraic group, is the identity component in the Hausdorff topology of the $\mathbb{R}$-points of an adjoint simple affine algebraic group.} Then $F$ is easily seen to be Zariski-dense in $G$. 

\begin{example}
\label{ex bounded distance}
    Let $\Gamma \coloneqq F$, and define a map $f \colon \Gamma \to G$ so that $f(g) = g$ for all $g \neq 1$ and $f(1) \neq 1$. Then $f(\Gamma)$ is unbounded and generates a group containing $F$, and is therefore Zariski-dense. Since $f$ is not a homomorphism,  \cref{simple case} implies that $f$ cannot be a quasihomomorphism.
\end{example}

The above example demonstrates how delicate the notion of a quasihomomorphism is; even taking a true homomorphism and changing a single value can result in a map that is not a quasihomomorphism.

\medskip

Next, we exhibit a counterexample to \cite[Theorem 4.28]{rank1}\footnote{Theorem 4.24 in the arXiv version}. The statement of the latter result is as follows. Let $M$ be a closed negatively curved manifold, and $x_1, \ldots, x_n \in \pi_1(M)$ be primitive elements that pairwise do not have conjugate powers. Let $G$ be a non-abelian connected Lie group and let $g_1, \ldots, g_n \in G$. Then it is claimed that there exists a homogeneous quasihomomorphism $\pi_1(M) \to G$ taking $x_i$ to $g_i$. Recall that a quasihomomorphism is {\em homogeneous} if its restriction to every cyclic subgroup is a homomorphism. We exhibit a counterexample in one of the simplest cases: where $M$ is a closed surface of genus two and $G = \mathrm{PSL}_2(\mathbb{R})$.

\begin{example}
\label{ex BV24}
    Let $F$ be a discrete free subgroup of $G=\mathrm{PSL}_2(\mathbb{R})$ of rank $4$, and let $g_1, h_1, g_2, h_2 \in F$ form a free basis. Let $\Gamma$ be the fundamental group of a closed surface of genus~$2$, and let $a_1, b_1, a_2, b_2$ be the standard generating set. Then the elements $a_i, b_i$ are primitive and pairwise do not have conjugate powers, since they are independent in the abelianisation of $\Gamma$. However there is no unbounded (let alone homogeneous) quasihomomorphism $\Gamma \to G$ mapping $a_i \to g_i, b_i \to h_i, i = 1, 2$. Indeed, the image of $\Gamma$ under a hypothetical such quasihomomorphism $f$ would generate a group containing the Zariski-dense subgroup $F$ of $G$, and therefore $f$ would necessarily be a homomorphism by \cref{simple case}, violating the defining relation $[a_1, a_2][b_1, b_2] = 1$.
\end{example}

\begin{remark}
\label{rem BV24}
    In addition to  the above straightforward example, notice that unipotent groups have no non-trivial compact quotients and hence \cref{no compact factors} provides a contradiction to the claim that there exist quasihomomorphisms $\Gamma \to G$ that are not constructible in the sense of Fujiwara--Kapovich when G is a non-abelian unipotent real algebraic group (cf. \cite[Theorem 4.31]{rank1}\footnote{Theorem 4.27 in the arXiv version}).
\end{remark}

\section{Real algebraic groups}
\label{s:lag}

In this section we collect some general facts about real algebraic groups that will be used in the proof of the general case. We refer the reader to \cite{borel} for general background. Given a subset $X \subset G$, we denote by $Z_G(X)$ the centraliser of $X$, and by $N_G(X)$ the normaliser of $X$ in $G$.

\begin{lemma}
\label{nested intersection}
    Let $G$ be a real algebraic group. Then any nested intersection of algebraic subgroups of $G$ eventually stabilises.
\end{lemma}

\begin{proof}
    Let $G_1 \supset G_2 \supset \cdots$ be a nested intersection of algebraic subgroups of $H$. Eventually, the real dimension stabilises. Since a top-dimensional subgroup of a real Lie group is open, and since real algebraic groups have only finitely many connected components (see the appendix of \cite{mostow}), we conclude.
\end{proof}

\begin{lemma}
\label{centralisers}
    Let $X \subset G$ be a subset of a real algebraic group. Then there exists a finite subset $Y \subset X$ such that $Z_G(X) = Z_G(Y)$.
\end{lemma}

\begin{proof}
    Because $X$ is a subspace of the separable metrisable space $G$, we have that~$X$ is itself separable. Let $\{x_1, x_2, \ldots \}$ be a countable dense subset of $X$. Then $Z_G(X) = Z_G(\{x_1, x_2, \ldots\})$. This shows that $Z_G(X)$ is the nested intersection of the algebraic subgroups $Z_G(\{x_1, \ldots, x_n\})$, $ n \geq 1$, so we conclude by \cref{nested intersection}.
\end{proof}

Let $G = R \ltimes U$, where $U$ is unipotent and $R$ is reductive, and let $\frakg, \fraku, \frakr$ denote the Lie algebras of $G, U, R$, respectively. We also denote by $Z$ the centre of $U$ and by $\frakz$ the Lie algebra of $Z$. Recall that the exponential map $\exp \colon \fraku \to U$ is a homeomorphism. The following extends this slightly.

\begin{prop}
\label{exp local homeo}
    There exists an open neighbourhood $0 \in \omega \subset \frakg$ such that $\exp$ restricted to $\omega + \fraku$ is a homeomorphism onto its image.
\end{prop}

The rest of this section is devoted to the proof of \cref{exp local homeo}, and is independent from the rest of the paper.

\begin{lemma}
\label{BCH consequence}
    For all sufficiently small $r \in \frakr$ there exists a linear isomorphism $\varphi_r \colon \frakz \to Z$ such that
    \[\exp(r + u + z) = \exp(r + u)\varphi_r(z)\]
    for all $u \in \fraku, z \in \frakz$.
\end{lemma}

\begin{proof}
    We set
    \[\varphi_r(z) \coloneqq \exp(-(r+u))\exp(r+u+z).\]
    The main claim is that $\varphi_r(z)$ indeed depends only on $r$ and $z$. Let us start by showing how this is enough to conclude. For $z_1, z_2 \in \frakz$, we have:
    \begin{align*}
    \exp(r) \varphi_r(z_1 + z_2) &= \exp(r + z_1 + z_2) = \exp(r + z_1) \varphi_r(z_2) \\&= \exp(r) \varphi_r(z_1)\varphi_r(z_2).
    \end{align*}
    So $\varphi_r$ is additive, and since  $\varphi_r$ is continuous, it follows that $\varphi_r$ is linear. Moreover, since $\exp$ is a local homeomorphism, for small enough $r$, the map $z \mapsto \exp(r+z)$ is locally injective, hence so is $\varphi_r$, so that $\varphi_r$ must therefore be a linear isomorphism.

    It remains to show that $\varphi_r(z)$ is independent of $u$. Recall the Baker--Campbell--Hausdorff formula \cite[Section 5]{hall:book}. Let
    \[B(X,Y) =  \sum_{n = 1}^\infty\frac {(-1)^{n-1}}{n}
    \sum_{\begin{smallmatrix} t_1 + s_1 > 0 \\ \vdots \\ t_n + s_n > 0 \end{smallmatrix}}
    \frac{[ X^{t_1} Y^{s_1} X^{t_2} Y^{s_2} \dotsm X^{t_n} Y^{s_n} ]}{\left(\sum_{j = 1}^n (t_j + s_j)\right) \cdot \prod_{i = 1}^n t_i! s_i!},\]
    where
    \[ [ X^{t_1} Y^{s_1} \dotsm X^{t_n} Y^{s_n} ] = [ \underbracket{X,\dotsm[X}_{t_1} ,[ \underbracket{Y,\dotsm[Y}_{s_1} ,\,\dotsm\, [ \underbracket{X,\dotsm[X}_{t_n} ,[ \underbracket{Y,\dotsm Y}_{s_n} ]]\dotsm].\]
    We set $m_n \coloneqq \sum (t_j + s_j)$, which is the number of occurrences of $X$ and $Y$ in the above expression. The BCH formula states that
    \[\exp(X)\exp(Y) = \exp(B(X, Y))\]
    whenever $B(X, Y)$ converges. Note that $B(X, Y)$ indeed converges for sufficiently small $X, Y$.
    
    Setting $X = -(r+u)$ and $Y = (r+u+z)$, we need to first show that the sum $B(X, Y)$ converges for all sufficiently small $r \in \frakr$ and for all $u \in \fraku, z \in \frakz$. Each iterated commutator has the following form:
    \begin{equation}
    \label{eq:BCH}
    \pm[ (r+u)^{t_1} (r+u+z)^{s_1} \dotsm (r+u)^{t_n} (r+u+z)^{s_n} ].
    \end{equation}
    Develop this expression using multilinearity, by separating all occurrences of $r$ from occurrences of $u$ and $(u+z)$. Because $\fraku$ is a nilpotent ideal, there is a uniform bound on the number of $u$ or $u+z$ terms that appear in such expression, unless it vanishes. Hence we obtain a bound of the form
    \[\|[ (r+u)^{t_1} (r+u+z)^{s_1} \dotsm (r+u)^{t_n} (r+u+z)^{s_n} ]\| \leq (C_{u, z}\|\ad(r)\|)^{m_n},\]
    where $C_{u, z}$ is a constant that depends only on the nilpotency class of $\fraku$ and the norms $\|u\|$ and $\|u+z\|$. This proves convergence, and hence the validity of the BCH formula, for small $r \in \frakr$ independently of $u, z$, as desired.
    
    Now let us study \cref{eq:BCH} further. When $m_n = 1$, there are only two terms: $-r-u$ and $r+u+z$. They sum to $z$, and in particular their sum is independent of $u$. Therefore we need to show that when $m_n \geq 2$, the value of \cref{eq:BCH} is independent of $u$.
    
    So now assume that $m_n \geq 2$. Up to grouping terms together, we may assume that $t_1, s_1 \neq 0$. Once again, we develop the expression using multilinearity, this time separating all occurrences of $(r+u)$ from occurrences of $z$. Up to applying the Jacobi identity, it suffices to show that
    \[[r+u,\cdots,[z, \cdots ]] = [r,\cdots,[z, \cdots ]],\]
    that is, in an iterated commutator where all terms are either $(r+u)$ or $z$, at least one $z$ appears, and the first term is $(r+u)$, we may replace the first term with $r$. Indeed, this is because $\frakz$ is an ideal, and $[r+u, z'] = [r,z']$ for all $z' \in \frakz$.
\end{proof}

\begin{proof}[Proof of \cref{exp local homeo}]
We proceed by induction on the dimension of $U$. For the base case, if $U$ is trivial, the statement is simply the fact that $\exp$ is a local homeomorphism.

Now suppose that $U$ has positive dimension, hence so does its centre $Z$. Let $\pi \colon G \to G/Z$ denote the projection. Applying the induction hypothesis on $G/Z$, we obtain an open neighbourhood $0 \in \omega \subset \frakg$ such that $\exp \colon \frakg/\frakz \to G/Z$ restricted to $d\pi(\omega) + \fraku/\frakz$ is a homeomorphism onto its image. We may also assume that $\omega$ is small enough that \cref{BCH consequence} holds for all $r \in \frakr \cap \omega$. We are left to prove that $\exp \colon \omega + \fraku \to \exp(\omega)U$ is a homeomorphism.

Let $v_1, v_2 \in \omega + \fraku$ be such that $\exp(v_1) = \exp(v_2)$. Then $d\pi(v_1) = d\pi(v_2)$, so there exists $z \in \frakz$ such that $v_2 = v_1 + z$. We write $v_1 = r + u$, for $r \in \frakr \cap \omega$. Hence $v_2 = r + u + z$, and by \cref{BCH consequence} we have $\exp(v_2) = \exp(v_1) \varphi_r(z)$. Because $\varphi_{r}$ is injective, $z = 0$, that is $v_1 = v_2$.

Let $g \in \exp(\omega)U$. By induction there exists $v = r + u \in (\frakr \cap \omega) + \fraku$ such that $\pi(g) = \exp(d\pi(v))$. So there exists $z \in Z$ such that $g = \exp(v)z$. Choose $z' \in \frakz$ such that $\varphi_r(z') = z$. Then
\[g = \exp(r+u)\varphi_r(z) = \exp(r+u+z) \in \exp(\omega + \fraku).\qedhere\]
\end{proof}

\section{Proof of the main theorem}
\label{s:proof}

In this section we prove the main theorem in its various versions (\cref{main theorem}, \cref{main theorem strong}, \cref{h2b} and \cref{no compact factors}). Let $f \colon \Gamma \to G$ be a quasihomomorphism, where $G$ is a real algebraic group. Let $D = D_f$ denote the defect set, and $\Delta = \Delta_f$ the defect group, i.e. the closure of the group generated by $D$.

\subsection{Setup}

\begin{lemma}
\label{connected}
    Up to bounded modification, $G$ is Zariski-connected, and $\langle f(\Gamma) \rangle$ is Zariski-dense.
\end{lemma}

This will be our only bounded modification, until \cref{final modification}. We also remark that this is the only step where we need to pass to a finite-index subgroup of $\Gamma$.

\begin{proof}
    Replacing $G$ by the Zariski-closure of $\langle f(\Gamma)\rangle$ is clearly a bounded modification. Next, we will show that there exists a bounded modification that takes values in the Zariski-connected component of $G$. By \cref{nested intersection}, we can repeat this process until both properties hold at once.
    
    Let $1 \in G_0, G_1, \ldots, G_n$ denote the Zariski-connected components of $G$. Let $X_i = f^{-1}(G_i \cap f(\Gamma)) \subset \Gamma$. By \cite[Lemma 7]{density:approx:lattices}\footnote{The cited lemma is for locally compact second countable groups, but this hypothesis is only used for the compactness of the Chabauty topology on the set of closed subsets, which holds in the discrete case without any cardinality assumptions.}, there exists $i \in \{0, \ldots, n\}$ and $m \geq 1$ such that $(X_i X_i^{-1})^m$ is a finite-index subgroup $\Gamma'$ of $\Gamma$.
    Now, for an element $x \in \Gamma'$ choose an expression $x = y_1 z_1^{-1} \cdots y_m z_m^{-1}$, where $y_j, z_j \in X_i$. Define
    \[f'(x) = f(y_1) f(z_1)^{-1} \cdots f(y_m) f(z_m)^{-1}\]
    and notice that $f(y_j) f(z_j)^{-1} \in G_i G_i^{-1} \subset G_0$, so $f'(\Gamma') \subset G_0$. That $f'$ is a quasihomomorphism close to $f|_{\Gamma'}$ follows from \cref{long product}. Therefore $f'$ is a bounded modification of $f$ that takes values in $G_0$.
\end{proof}

Now $G$ is a Zariski-connected real algebraic group, so it decomposes as $G = (R/F) \ltimes U$, where $U$ is unipotent, $R$ is reductive, and $F$ is a finite subgroup of $R$. The action of $R/F$ on $U$ pulls back to an action of $R$, which defines a semidirect product $R \ltimes U$, which surjects on $H$ with kernel $F$. By \cref{compact extension}, we can therefore assume that
\[G = R \ltimes U \quad \text{with }R \text{ connected reductive and } U \text{ unipotent}.\]

We now decompose the reductive part as
\[R = G_b \times G_h \times T,\]
where
\begin{itemize}
\item $T$ is a torus;
\item $G_b$ is the product of all those simple factors $S$ such that the composition $\Gamma \xrightarrow{f} G \to S$ is bounded but does not have central defect;
\item $G_h$ is the product of all those simple factors $S$ such that the composition $\Gamma \xrightarrow{f} G \to S$ has central defect.
\end{itemize}

By \cref{simple case}, every simple factor is indeed either in $G_b$ or in $G_h$. We also note that the projection of $\langle D \rangle$ to $G_h$ is trivial, hence so is the projection of $\Delta$ to $G_h$.

\subsection{The conjugacy action on \texorpdfstring{$\Delta$}{Δ}}

Recall from \cref{BOL D} that for all $x \in \langle D \rangle$, the orbit $x^{f(\Gamma)}$ is bounded. It is a priori not clear that this extends to $x \in \Delta$, or to $x \in \zar{\Delta}$.  Note that this is a feature unique to our setting and does not occur when the target is discrete, as in \cite{fujiwarakapovich}. Our next goal is to tackle $\Delta$.

\begin{lemma}
\label{BOL Delta}

For every identity neighbourhood $W \subset \Delta$, there exists an identity neighbourhood $W_0 \subset W$ such that $W_0^{f(\Gamma)} \subset W$. Moreover, for every $x \in \Delta$, the set $x^{f(\Gamma)}$ is bounded.
\end{lemma}

Again, the ``moreover'' part holds for $x \in \langle D \rangle$, by \cref{BOL D}. Indeed, choosing $W$ to be a compact identity neighbourhood of $\Delta$, \cref{BOL Delta} gives the same result for the corresponding $W_0$. So $x^{f(\Gamma)}$ is bounded for all $x \in W_0 \langle D \rangle = \Delta$, as claimed. Hence we can focus on the first statement.

\medskip

Using the same notation as in \cref{exp local homeo}, we fix an open neighbourhood $0 \in \omega \subset \frakg$ such that $\exp$ restricted to $\omega + \fraku$ is a homeomorphism onto its image.

\begin{claim}
\label{BOL algebra}
    There exists an open neighbourhood $0 \in \omega' \subset \frakg$ such that
\[\Ad(f(\Gamma)) (\exp^{-1}(\Delta) \cap \omega' + \fraku) \subset \omega + \fraku.\]
\end{claim}

\begin{proof}
    Throughout the proof, we work modulo $\fraku$, hence in the Lie algebra $\frakr$ of $R$. Recall that $R$ splits as $G_b \times G_h \times T$, that the projection of $f(\Gamma)$ onto $G_b$ is bounded, the projection of $\Delta$ onto $G_h$ is finite and central, and $T$ is abelian. The Lie algebra $\frakr$ has an $\mathrm{Ad}$-invariant splitting as the direct sum of the Lie algebras of the three factors, and we study the adjoint action coordinate-wise.
    
    Because $T$ is central in $R$, we have that $T$ is in the kernel of the adjoint action. Because $\Delta$ has finite central projection onto $G_h$, the preimage $\exp^{-1}(\Delta)$ has trivial $\frakg_h$ component: indeed, non-trivial central elements of the semi-simple group $G_h$ cannot be in the image of the exponential map.  So the adjoint action of $f(\Gamma)$ on $\exp^{-1}(\Delta)$ reduces to the adjoint action of the projection of $f(\Gamma)$ onto $G_b$, which is bounded by definition. Hence $\Ad(f(\Gamma))$ is a uniformly bounded family of operators of $\exp^{-1}(\Delta)$.  It follows that the desired open neighbourhood $\omega'$ of $0$ exists.
\end{proof}

Define
\[\frakb \coloneqq \{ v \in \frakg : \Ad(f(\Gamma))(v) \text{ is bounded}\}.\]

\begin{claim}
\label{Delta in b}

$\Delta \cap \exp(\omega' + \fraku) \subset \exp(\frakb)$.
\end{claim}

\begin{proof}
	Suppose first that $x \in \langle D \rangle \cap \exp(\omega' + \fraku)$. Let $v \in \omega' + \fraku$ be such that $\exp(v) = x$. Then
    \[x^{f(\Gamma)} = \exp(\Ad(f(\Gamma))(v)).\]
    By \cref{BOL algebra}:
    \[\Ad(f(\Gamma))(v) \subset \omega + \fraku.\]
    Because $\exp|_{\omega + \fraku}$ is a homeomorphism onto its image, and $x^{f(\Gamma)}$ is bounded by \cref{BOL D}, $\Ad(f(\Gamma))(v)$ must be bounded as well, that is $v \in \frakb$.

    Now for the general case, suppose that $x \in \Delta \cap \exp(\omega' + \fraku)$, and let $x_n \in \langle D \rangle$ be such that $x_n \to x$. Up to discarding finitely many terms, we may assume that $x_n \in \exp(\omega' + \fraku)$, so by the previous case we can choose $v_n \in (\omega' + \fraku) \cap \frakb$ such that $\exp(v_n) = x_n$. Because $\exp$ is a homeomorphism on $\omega' + \fraku$, we have that $v_n$ converges to some $v \in \omega' + \fraku$ such that $\exp(v) = x$. Finally, since $\frakb$ is a vector subspace of $\frakg$ and is in particular closed, we have that $v_n \in \frakb$ implies $v \in \frakb$.
\end{proof}

\begin{proof}[Proof of \cref{BOL Delta}]
	Let $\frakd$ denote the Lie algebra of $\Delta$. \cref{Delta in b} implies that $\frakd \subset \frakb$. Let $W$ be an identity neighbourhood of $\Delta$, so $\exp^{-1}(W)$ is an open subset of $\frakd$. Because $\Ad(f(\Gamma))$ is a bounded family of operators of $\frakb$, hence of $\frakd$, there exists an open subset $\omega_0$ of $\frakd$ such that $\Ad(f(\Gamma))(\omega_0) \subset \exp^{-1}(W)$. Taking $W_0 = \exp(\omega_0)$, we have
	\[W_0^{f(\Gamma)} = \exp(\Ad(f(\Gamma))(\omega_0) \subset W. \qedhere\]
\end{proof}

\subsection{The conjugacy action on \texorpdfstring{$\zar{\Delta}$}{Δzar}}

In this subsection we strengthen and extend \Cref{BOL Delta} to apply to the Zariski-closure $\zar{\Delta}$.

\begin{prop}
\label{BOL}
    There exists a compact set $K \subset G$ such that $f(\Gamma) \subset Z_G(\zar{\Delta})K$. In particular, for all $x \in \zar{\Delta}$, the set $x^{f(\Gamma)}$ is bounded.
\end{prop}

By definition $\Delta = \overline{\langle D \rangle}$ is the closure of a group generated by a bounded set and is hence compactly generated. So we can apply \cite{hochschild} to see that the group of automorphisms $\Aut(\Delta)$ with the compact open topology is a Lie group. A set of the form
\[\mathcal{N}(C, W) \coloneqq \{ \varphi \in \Aut(\Delta) : \varphi(x) \in xW \text{ for all } x \in K \},\]
for $C$ a compact set and $W$ an open identity neighbourhood, is an identity neighbourhood.

We now consider the map
\[\alpha \colon N_G(\Delta) \to \Aut(\Delta)\]
given by the conjugacy action.

\begin{claim}
\label{fGamma bounded}
    The set $\alpha(f(\Gamma))$ is bounded in $\Aut(\Delta)$.
\end{claim}

\begin{proof}
    We verify this using the Arz{\'e}la--Ascoli Theorem. First, for every $x \in \Delta$, the second part of \cref{BOL Delta} gives that $\alpha(f(\Gamma))(x) = x^{f(\Gamma)}$ is bounded. Moreover, $\alpha(f(\Gamma))$ is equicontinuous: because this is a group topology, it suffices to verify this locally, which is given by the first part of \cref{BOL Delta}.
\end{proof}

\begin{claim}
\label{zimmer consequence}
    There exists a finite subset $F \subset \Delta$, an open identity neighbourhood $W \subset \Delta$, and a compact set $K \subset G$ such that
    \[\alpha^{-1}(\mathcal{N}(F, W)) \subset Z_G(\zar{\Delta})K.\]
\end{claim}

\begin{proof}
    Because centralising is an algebraic condition, $Z_G(\Delta) = Z_G(\zar{\Delta})$. By \cref{centralisers} we can choose $F = \{x_1, \ldots, x_n\} \subset \Delta$ to be such that $Z_G(F) = Z_G(\Delta)$.

    Consider the action of $G$ on the product $G^n$ by coordinate-wise conjugacy, and let $O$ denote the orbit of $(x_1, \ldots, x_n)$. By \cite[Theorem 3.1.3]{zimmer}, the action has locally closed orbits. In particular, there exists a compact neighbourhood $W$ of $e$ in $G$ such that
    \[O \cap (x_1W, \ldots, x_nW)\]
    is compact in $O$. Moreover, by \cite[Theorem 2.1.14]{zimmer} the orbit map
    \[G/Z_G(F) \to O\]
    is a homeomorphism. Therefore the preimage of $O \cap (x_1W, \ldots, x_nW)$ is compact in $G/Z_G(F)$, and we can write its lift to $G$ as $Z_G(F) K$, for a compact subset $K \subset G$.
    
    Now if $g \in \alpha^{-1}(\mathcal{N}(F, W))$, then for all $i = 1, \ldots, n$, we have
    \[x_i^g = \alpha(g)(x_i) \in x_iW,\]
    that is $g$ lies in the preimage of $O \cap (x_1W, \ldots, x_nW)$ under the orbit map, and so it belongs to $Z_G(F) K = Z_G(\zar{\Delta}) K$ by the above.
\end{proof}

\begin{proof}[Proof of \cref{BOL}]
    Let $F, W$ be as in \cref{zimmer consequence}. By \cref{fGamma bounded}, there exist finitely many elements $\gamma_1, \ldots, \gamma_n \in \Gamma$ such that
    \[\alpha(f(\Gamma)) \subset \bigcup_i \mathcal{N}(F, W) \alpha(f(\gamma_i),\]
    hence
    \[f(\Gamma) \subset \bigcup_i \alpha^{-1}(\mathcal{N}(F, W)) f(\gamma_i).\]
    By \cref{zimmer consequence}, we therefore have
    \[f(\Gamma) \subset Z_H(\zar{\Delta}) K \{f(\gamma_1), \ldots, f(\gamma_n)\}.\]
    Since $K \{f(\gamma_1), \ldots, f(\gamma_n)\}$ is still compact, we conclude.
\end{proof}

\subsection{Endgame}
\label{endgame}

By \cref{BOL}, there exists a compact set $K$ such that $f(\Gamma) \subset Z_G(\zar{\Delta}) K.$ Define $N \coloneqq \zar{\Delta} Z_G(\zar{\Delta})$, which is an algebraic normal subgroup of $G$ containing the defect. Therefore the composition
\[\Gamma \xrightarrow{f} G \to G/N\]
is a homomorphism. Its image is a precompact Zariski-dense subgroup of the algebraic group $G/N$, which implies that $G/N$ itself is compact\footnote{This is the only step where we use the hypothesis that $G$ is a \emph{real} algebraic group, as opposed to a \emph{$p$-adic} algebraic group.}
In particular $N$ must contain the unipotent $U$, and so we can write
\[G = BN,\]
for $B$ a compact algebraic subgroup of $G$ contained in $R$.

Now for each $\gamma \in \Gamma$ we can choose a decomposition of the form
\[f(\gamma) = b(\gamma) z(\gamma) d(\gamma) \in B \cdot Z_G(\zar{\Delta}) \cdot \zar{\Delta}.\]
Because $f(\Gamma) \subset Z_G(\zar{\Delta}) K$, and $Z_G(\zar{\Delta})$ is normal, we may choose this expression so that $d(\Gamma)$ is bounded. Set $f'(\gamma) \coloneqq b(\gamma) z(\gamma) \in B Z_G(\zar{\Delta})$.

\begin{claim}
\label{final modification}
$f' \colon \Gamma \to B Z_G(\zar{\Delta})$ is a bounded modification of $f$. Moreover, $D_{f'} \subset \zar{\Delta}$.
\end{claim}

\begin{proof}
We have that $f'$ is close to $f$ because $d(\Gamma)$ is bounded. Recall from \cref{notation approx}, that for a bounded set $K$, we write $x \approx_K y$ if $x \in yK$. In the course of the proof, we write $\approx_f$ if $K = D_f$, $\approx_d$ if $K = d(\Gamma)$, and finally $\approx_{B, d}$ if $K = \{ bdb^{-1} : b \in B, d \in d(\Gamma) \}$. Note that each of these bounded sets is contained in $\zar{\Delta}$.

Now we estimate:
\begin{align*}
	f'(\gamma \delta) &\approx_d f(\gamma \delta) \approx_f f(\gamma) f(\delta) \approx_d f(\gamma) f'(\delta) \\
	&= f'(\gamma) d(\gamma) b(\delta) z(\delta) = f'(\gamma) b(\delta) d(\gamma)^{b(\delta)} z(\delta) \\
	&= f'(\gamma) b(\delta) z(\delta) d(\gamma)^{b(\delta)} \approx_{B, d} f'(\gamma) f'(\delta).
\end{align*}
Here we used that $\zar{\Delta}$ is normal; hence $d(\gamma)^{b(\delta)} \subset \zar{\Delta}$ commutes with $z(\delta)$.

Finally, as we remarked above, each compact set appearing in an approximate equality $\approx$ is contained in $\zar{\Delta}$, so $D_{f'} \subset \zar{\Delta}$.
\end{proof}

\begin{proof}[Proof of \cref{main theorem}]
    By \cref{final modification}, we may consider
    \[f' \colon \Gamma \to B Z_G(\zar{\Delta}) = Z_G(\zar{\Delta}) B,\]
    where $B$ is a compact (algebraic) subgroup, and $D' \coloneqq D_{f'}$ is contained in $\zar{\Delta}$. Therefore
    \[(D')^G = (D')^{Z_G(\zar{\Delta})B} = (D')^B \subset \zar{\Delta}.\]
    So $(D')^B$ is a normal bounded subset of $G$ containing the defect $D'$.
\end{proof}

\begin{proof}[Proof of \cref{no compact factors}]
    Suppose that $\langle f(\Gamma) \rangle$ is Zariski-dense in $G$, and $G$ has no non-trivial compact quotient. Then $G$ is automatically Zariski-connected, therefore the bounded modification from \cref{connected} is not needed in this case. Now with the notation above, we have that $B$ is a compact quotient of $G$, so $B$ must be trivial, which implies that the bounded modification from \cref{final modification} takes values in the centraliser of the defect.
\end{proof}

\begin{proof}[Proof of \cref{main theorem strong}]
    As in the proof of \cref{main theorem}, we have a bounded modification $f' \colon \Gamma \to B Z_G(\zar{\Delta})$, where $B$ is a compact (algebraic) subgroup, and $D_{f'}$ is contained in $\zar{\Delta}$. The action of $G'$ by conjugacy on $\zar{\Delta}$ factors through $B$, hence the action of $\zar{\Delta}$ by conjugacy on itself factors through a compact group~$C_0$. This implies that
    \[\zar{\Delta} = Z_{\Delta} \times C_0,\]
    where $Z_{\Delta}$ is the centre of $\zar{\Delta}$, and $C_0$ is a compact group. It follows that there exists a unique maximal compact (algebraic) subgroup $C < \zar{\Delta}$, which is characteristic in $\zar{\Delta}$ hence normal in $G'$. Quotienting by $C$, the image of $\zar{\Delta}$ is a simply connected abelian normal subgroup $A \lhd G'/C$, and the action by conjugacy of $G'$ on $A$ factors through the compact group $B$, so that $A$ is rigid in $G'/C$.
\end{proof}

\begin{proof}[Proof of \cref{h2b}]
    Let us start by noticing that if $H^2_b(\Gamma; \pi)$ vanishes for every finite-dimensional orthogonal representation $\pi$ with no non-trivial invariant vectors, and vanishes when $\pi$ is the trivial representation, then $H^2_b(\Gamma; \pi)$ vanishes for every finite-dimensional orthogonal representation $\pi$ \cite[Lemma 1.2.10]{monod:book}.
    
    Thanks to \cref{main theorem strong} and \cref{compact extension}, we may reduce to the case in which $f \colon \Gamma \to G$ is such that $D_f$ is contained in a rigid abelian algebraic normal subgroup $A \lhd G$, such that $D = D_f$ is contained in $A$. We will show that, under the assumption that bounded cohomology vanishes with coefficients in every finite-dimensional orthogonal representation, there exists a bounded map $b \colon \Gamma \to A$ such that $\gamma \mapsto f(\gamma) b(\gamma)$ is a homomorphism. This will conclude the proof of the second statement, since the map $\gamma \mapsto f(\gamma) b(\gamma)$ is a bounded modification of $f$.

    The composition $\Gamma \xrightarrow{f} G \to G/A$ is a homomorphism, therefore defines an orthogonal representation $\pi$ of $\Gamma$ on $A$. Now consider the lifting problem
    \[\begin{tikzcd}
	& G \\
	\Gamma \\
	& {G/A}
	\arrow[from=1-2, to=3-2]
	\arrow[dashed, from=2-1, to=1-2]
	\arrow[from=2-1, to=3-2]
    \end{tikzcd}\]
    The set-theoretic lift $f \colon \Gamma \to G$ defines the bounded $2$-cocycle
    \[(\gamma_1, \gamma_2) \mapsto f(\gamma_2)^{-1} f(\gamma_1)^{-1} f(\gamma_1 \gamma_2) \in A,\]
    which in turn represents a class in $H^2_b(\Gamma; \pi)$. The stronger vanishing hypothesis gives a bounded primitive $b \colon \Gamma \to A$ for this cocycle, which means exactly that $\gamma \mapsto f(\gamma) b(\gamma)$ is a homomorphism.

    Suppose instead that we only have vanishing for representations with no non-trivial invariant vectors. Then $H^i_b(\Gamma; (\pi^\Gamma)^\perp)$ vanishes for $i = 1$ (this is always true \cite[Proposition 6.2.1]{monod:book}) and $i = 2$, so the long exact sequence in bounded cohomology gives an isomorphism $H^2_b(\Gamma; \pi^\Gamma) \to H^2_b(\Gamma; \pi)$ induced by the  change of coefficients map. That is, every bounded $2$-cocycle is boundedly cohomologous to one that takes values in the invariants. Therefore the same argument as above shows that there exists a bounded map $b \colon \Gamma \to A$ such that $\gamma \mapsto f(\gamma) b(\gamma)$ has central defect.
\end{proof}

\section{Counterexample to central defect}
\label{s:cex}

\begin{prop}
\label{counterexample central defect}
    Let $\Gamma$ be an acylindrically hyperbolic group that surjects onto $\Z$. Let $G = \mathrm{SO}(2) \ltimes \RR^2$. Then there exists a quasihomomorphism $f \colon \Gamma \to G$ such that no bounded modification of $f$ has central defect.
\end{prop}

Acylindrically hyperbolic groups include non-abelian free groups, and more generally all non-elementary (relatively) hyperbolic groups \cite{osin}.

We use again the exponential notation for conjugacy; note that for $v \in \RR^2$ and $A \in \mathrm{SO}(2)$, the conjugacy in $G$ written as $v^A$ is just the ordinary right action of $A$ on $v$.

\begin{proof}
    The map to $\Z$ induces a homomorphism $u \colon \Gamma \to \mathrm{SO}(2)$ with dense image. Recall that a \emph{$u$-quasicocycle} is a map $c \colon \Gamma \to \RR^2$ such that
    \[D_c \coloneqq \{ c(gh) - (c(g)^{u(h)} + c(h)) : g, h \in \Gamma\}\]
    is bounded. If $D_c$ is trivial, we call $c$ a \emph{$u$-cocycle}. We can choose a $u$-quasicocycle $c$ that is not close to any $u$-cocycle. By \cite[Theorem 1.1]{quasicocycles}, this is possible provided that there exists an invariant vector for the maximal finite normal subgroup of $\Gamma$; in our case, we chose an action that factors through $\Z$, hence every finite subgroup of $\Gamma$ in fact acts trivially.

    Define $f = (u, c) \colon \Gamma \to \mathrm{SO}(2) \ltimes \mathbb{R}^2$. This is a quasihomomorphism:
    \begin{align*}
        f(gh) &= (u(gh), c(gh)) \approx_{D_c} (u(g)u(h), c(g)^{u(h)} + c(h)) \\
        &=(u(g), c(g))(u(h), c(h)) = f(g)f(h).
    \end{align*}
    Suppose by contradiction that there exists a bounded modification $f' \colon \Gamma' \to G'$ such that $G'$ centralises $D_{f'}$. Note that $u|_{\Gamma'} \colon \Gamma' \to \mathrm{SO}(2)$ is still a homomorphism with dense image, and $c|_{\Gamma'}$ is a $u|_{\Gamma'}$ quasicocycle. Moreover, it follows from the interpretation of quasicocycles in terms of bounded cohomology, and the injectivity of the restriction to $\Gamma'$ \cite[Proposition 8.6.2]{monod:book}, that $c|_{\Gamma'}$ is not close to any $u|_{\Gamma'}$-cocycle. Therefore, we may assume that $\Gamma' = \Gamma$, and for clarity we will denote $f' = \varphi$, and $G' = H$. By assumption, there exists a bounded set $D_{\varphi, f}$ such that $\varphi (g) \approx_{D_{\varphi, f}} f(g)$ (see \cref{notation approx}).

    \emph{Suppose first that $H$ abelian.} Then $\varphi$ is a quasimorphism with values in an abelian group, and in particular $\varphi$ is uniformly bounded on commutators, more precisely $\varphi([g, h])$ is bounded by twice the defect $D_{\varphi}$ \cite[Section 2.2.3]{calegari}.
    \begin{align*}
        0 &\approx_{2D_\varphi} \varphi([g, h]) \approx_{D_{\varphi, f}} f([g, h]) = c([g, h]) = c(g^{-1}h^{-1}gh) \\
        &\approx_{D_c} c(g^{-1})^{u(h^{-1}gh)} + c(h^{-1})^{u(gh)} + c(g)^{u(h)} + c(h).
    \end{align*}
    Now
    \begin{align*}
        0 &\approx_{D_f} f(gg^{-1}) = c(g g^{-1}) \approx_{D_c} c(g)^{u(g^{-1})} + c(g^{-1}) \\
        \Rightarrow c(g^{-1}) &\approx_{D_c \cdot D_f} -c(g)^{u(g^{-1})}.
    \end{align*}
    Combining the above, together with the fact that $u$ is a homomorphism to an abelian group, we get a bounded set $K$ such that
    \begin{align*}
        0 &\approx_K -c(g) - c(h)^{u(g)} + c(g)^{u(h)} + c(h) \\
        \Rightarrow c(g)^{I - u(h)} &\approx_K c(h)^{I - u(g)}.
    \end{align*}
    Fixing $g$ and letting $h$ vary, the left hand side remains bounded in terms of $g$, which implies that $c(\Gamma)$ is uniformly close to the fixed point set of $u(g)$, which is $\{0\}$ as soon as $u(g)$ is non-trivial. Therefore $c(\Gamma)$ is bounded, contradiction.

    \emph{Now suppose that $H$ is not abelian.} Then $H$ cannot be contained in $\RR^2$, so there exists $(x, a) \in H$ with $x \neq I$. The centraliser of $(x, a)$ intersects $\RR^2$ trivially, which implies that $D_\varphi$ intersects $\RR^2$ trivially. Similarly, $H$ cannot project injectively to $\mathrm{SO}(2)$, so there exists $(I, b) \in H$ with $b \neq 0$. The centraliser of $(I, b)$ is contained in $\RR^2$, and therefore $D_\varphi \subset \RR^2$. We conclude that $D_\varphi$ is trivial, i.e. that $\varphi$ is a homomorphism. So we can write $\varphi(g) = (v(g), d(g))$, where $v \colon \Gamma \to \mathrm{SO}(2)$ is a homomorphism, and $d \colon \Gamma \to \RR^2$ is a $v$-cocycle, i.e. $d(gh) = d(g)^{v(h)} + d(h)$. If $v = u$, then $d$ will be a $u$-cocycle, and $\varphi(g) \approx_{D_{\varphi, f}} f(g) \Rightarrow d(g) \approx_{D_{\varphi, f}} c(g)$, which contradicts our choice of $c$. Therefore, it remains to show that $v = u$.

    Write $\tilde{d}(g) = d(g)^{v(g)^{-1}}$. Then
    \[\tilde{d}(gh) = d(gh)^{v(gh)^{-1}} = (d(g)^{v(h)} + d(h))^{v(gh)^{-1}} = \tilde{d}(g) + \tilde{d}(h)^{v(g)^{-1}}.\]
    In words, $\tilde{d}$ is a \emph{left} $v$-cocycle, where we consider the left action of $\mathrm{SO}(2)$ on $\RR^2$ given by right multiplication with the inverse. Moreover,
    \[\varphi(g) = (v(g),d(g)) = \tilde{d}(g) v(g) \approx_{\mathrm{SO}(2)} \tilde{d}(g).\]
    Similarly, $\tilde{c}(g) = c(g)^{u(g)^{-1}}$ is a left $u$-quasicocycle, and $f(g) \approx_{\mathrm{SO}(2)} \tilde{c}(g)$. Therefore $\tilde{c}(g) \approx_{\mathrm{SO}(2)} \tilde{d}(g)$. We now estimate:
    \begin{align*}
        \tilde{c}(g) + \tilde{c}(h)^{u(g)^{-1}} &\approx_{D_{\tilde{c}}} \tilde{c}(gh) \approx_{\mathrm{SO}(2)} \tilde{d}(gh) \\
        &= \tilde{d}(g) + \tilde{d}(h)^{v(g)^{-1}} \approx_{\mathrm{SO}(2)} \tilde{c}(g) + \tilde{c}(h)^{v(g)^{-1}}.
    \end{align*}
    We have shown that there exists a bounded set $K$ such that
    \[\tilde{c}(h)^{u(g)^{-1}} \approx_K \tilde{c}(h)^{v(g)^{-1}}.\]
    Note that $\tilde{c}(\Gamma)$ is unbounded, because $f(\Gamma)$ is unbounded. Fixing $g$ and letting $h$ vary, we conclude that $v = u$, which finishes the proof.
\end{proof}

Of course, \cref{counterexample central defect} does not contradict \cref{main theorem}, because the defect of $f$ is contained in $\mathbb{R}^2$, and so its $\mathrm{SO}(2)$-closure is bounded and normal.

\bibliographystyle{halpha}
\bibliography{refs}

\end{document}